\documentclass[11pt,letterpaper]{amsart}
\usepackage{amssymb,amsmath,amsrefs}
\usepackage[colorlinks=true]{hyperref}
\hypersetup{urlcolor=blue, citecolor=green}
\usepackage[dvips]{graphicx}
\usepackage{color}
\usepackage{svg}
\usepackage{subcaption}
\usepackage{float}
\usepackage{shellesc}

\graphicspath{{Imagens/}}

\theoremstyle{plain}

\newtheorem{mainthm}{Theorem}

\newtheorem*{conj*}{Conjecture}
\newtheorem*{cor*}{Corollary}
\newtheorem{theorem}{Theorem}[section]

\newtheorem{prop}[theorem]{Proposition}
\newtheorem{proposition}[theorem]{Proposition}

\theoremstyle{definition}
\newtheorem*{def*}{Definition}
\newtheorem{remark}[theorem]{Remark}

\newtheorem{definition}[theorem]{Definition}

\newcommand{\N}{\mathbb{N}}
\newcommand{\R}{\mathbb{R}}
\newcommand{\eps}{\varepsilon}

\title[Holonomies of Anosov diffeomorphisms]{Conformal hyperbolic distances and holonomies of Anosov diffeomorphisms}
\author{Bernardo Carvalho\textsuperscript{*}}

\date{\today}

\begin{document}

\renewcommand{\thefootnote}{}



\footnote{(*): National Laboratory for Scientific Computing – LNCC/MCTI}

\renewcommand{\thefootnote}{\arabic{footnote}}
\setcounter{footnote}{0}

\begin{abstract}
We use conformal hyperbolic distances to analyze the stable/unstable holonomies of Anosov diffeomorphisms. The main result is that the lengths of the stable/unstable curves in a conformal hyperbolic distance are preserved by the stable/unstable holonomies, whenever they are well-defined. In particular, the existence of a conformal hyperbolic distance with well-defined length of stable/unstable curves ensures that the holonomies are globally defined. We exhibit a conformal hyperbolic distance with well-defined length of stable/unstable curves for each real Anosov diffeomorphism, which are Anosov diffeomorphisms with a one-dimensional finest dominated splitting where contraction/expansion in each space is constant along orbits.
\end{abstract}

\maketitle

\section{Introduction}

An \emph{Anosov diffeomorphism} is a diffeomorphism $f\colon M\to M$ of a closed, connected, and smooth manifold $M$ with a Riemannian metric $\|.\|$ satisfying: 
there is an invariant splitting $TM=E^s\oplus E^u$ and constants $c>0$ and $\lambda\in(0,1)$ such that 
$$\|Df^k_{|E^s}\|\leq c\lambda^k \,\,\, \textrm{and}
      \,\,\, \|Df^{-k}_{|E^u}\|\leq c\lambda^k \,\,\,\,\,\, \text{for every} \,\,\,\,\,\, k\in\N.$$ 
The transitivity of Anosov diffeomorphisms was conjectured by Stephen Smale in \cite{S} and has been intriguing dynamicists since then. From the classical works of John Franks \cite{F1} and \cite{F2} till more recent works of Victor Kleptsyn and Yury Kudryashov \cite{KK}, the idea of proving transitivity revolves around proving that stable/unstable holonomies can be infinitely extended. To define this precisely, we first define the \emph{stable set} of $x$ by 
$$W^s(x):=\{y\in M;\;d(f^n(y), f^n(x))\to 0 \,\,\,\,\,\, \text{when} \,\,\,\,\,\, n\to\infty\}$$
and the \emph{unstable set} of $x$ by 
$$W^u(x):=\{y\in M;\;d(f^n(y), f^n(x))\to 0 \,\,\,\,\,\, \text{when} \,\,\,\,\,\, n\to-\infty\},$$ where $d$ is the distance induced by the Riemannian metric, and note that on Anosov diffeomorphisms, these sets are leaves of a pair of transversal foliations, called stable and unstable foliations (this is a consequence of the Stable Manifold Theorem \cite{HPS}). We say that stable/unstable holonomies can be \emph{infinitely extended} (or are globally defined) if for every $x\in M$, $y\in W^s(x)$, and $z\in W^u(x)$, there exists a continuous map $g\colon[0,1]\times[0,1]\to M$ such that $g(0,0)=x$, $g(0,1)=y$, $g(1,0)=z$, $g(t,.)$ is a stable curve for every $t\in[0,1]$, and $g(.,t)$ is an unstable curve for every $t\in[0,1]$.

It was proved by M. Brin in \cite{B} that stable/unstable holonomies of Anosov diffeomorphisms satisfying a pinching condition can be infinitely extended. 
This includes the case of conformal Anosov diffeomorphisms, which are the Anosov diffeomorphisms where there exists a Riemannian metric $\|.\|$ satisfying
$$\|Df(x)v^s\|=\lambda\|v^s\| \,\,\,\,\,\, \text{and} \,\,\,\,\,\, \|Df^{-1}(x)v^u\|=\lambda\|v^u\|$$
for every $v^s\in E^s(x)$, $v^u\in E^u(x)$, and $x\in M$, since in this case the pinching condition is trivially satisfied. 

The reason usually cited in the literature for the infinite extendability of stable/unstable holonomies not being proved is the lack of regularity of stable and unstable foliations. Even though stable and unstable leaves have the same regularity as the diffeomorphism, the stable and unstable foliations of Anosov diffeomorphisms are (in general) only Hölder continuous (see \cite{HPS} and \cite{PSW}), and this seems to be not enough to ensure the existence of the map $g$ as above. The idea is that stable/unstable holonomies could ``blow up'' in finite time (see Figure 5 in \cite{F1}).

To overcome this obstacle, we shift our focus from Riemannian metrics to conformal hyperbolic distances, inspired by the self-similar hyperbolic distance introduced by Alfonso Artigue \cite{A} and based on the hyperbolic distance of Fathi \cite{F}. We use the word distance to distinguish it from a conformal Riemannian metric. By analyzing how the stable/unstable holonomies behave within this conformal framework, we obtain conditions on curves that prevent the holonomies from blowing up. The following is the first main result of this article:  

\begin{mainthm}\label{A}The lengths of the stable/unstable curves in a conformal hyperbolic distance are preserved by the stable/unstable holonomies, whenever they are well-defined.\end{mainthm}  

A direct consequence of Theorem \ref{A} is that the existence of a conformal hyperbolic distance with well-defined lengths for stable/unstable curves ensures that the stable/unstable holonomies can be infinitely extended. Since conformal hyperbolic distances are not necessarily induced by a Riemannian metric, establishing that a stable or unstable curve has a well-defined length (see Definition \ref{length}) can be non-trivial. To show the utility of Theorem \ref{A}, we apply it to the class of real Anosov diffeomorphisms, which are Anosov diffeomorphisms admitting a one-dimensional finest dominated splitting where the contraction/expansion in each sub-space is constant along orbits (see Section 4). For this class, we prove the following:

\begin{mainthm}\label{thmreal}If $f\colon M\to M$ is a real Anosov diffeomorphism, then it admits a conformal hyperbolic distance with well-defined lengths of stable/unstable curves. Consequently, its stable/unstable holonomies can be infinitely extended and $f$ is transitive.\end{mainthm}

This article is organized as follows: in Section \ref{Conformalself}, we discuss basic properties of conformal hyperbolic distances and conformal structures, and prove Theorem \ref{A}; and in Section \ref{Real}, we construct the specific conformal hyperbolic distance for real Anosov diffeomorphisms and prove Theorem \ref{thmreal}.


\section{Stable/unstable sets are conformal}\label{Conformalself}

In this section, we discuss local stable/unstable holonomies of topologically hyperbolic homeomorphisms using a conformal hyperbolic distance. We prove that the holonomies in this case are pseudo-isometric, following a similar result of Artigue \cite{A} which uses a self-similar hyperbolic metric. We introduce the concept of a conformal structure, a distance inside each (global) stable/unstable set where $f$ acts in a conformal way, and define a conformal structure using a conformal hyperbolic distance. We also discuss the hypothesis of existence of lengths for stable/unstable curves on a conformal hyperbolic distance and how these lengths behave under the pseudo-isometric local stable/unstable holonomies. We conclude proving Theorem \ref{A}, obtaining that holonomies preserve the lengths of stable/unstable curves in a conformal hyperbolic distance.

\vspace{+0.4cm}

\hspace{-0.4cm}\textbf{Conformal hyperbolic distances}

\vspace{+0.4cm}

We begin this section by recalling basic definitions of topological dynamics and properties of hyperbolic distances. 
Let $f\colon X\to X$ be a homeomorphism of a compact metric space $(X,d)$. 
For each $x\in X$ and $\eps>0$ we consider the \emph{$\eps$-stable set} of $x\in X$ as the set 
$$W^s_{\eps}(x):=\{y\in X; \,\, d(f^k(y),f^k(x))\leq\eps \,\,\,\, \textrm{for every} \,\,\,\, k\geq 0\}$$
and the \emph{$\eps$-unstable set} of $x$ as the set 
$$W^u_{\eps}(x):=\{y\in X; \,\, d(f^k(y),f^k(x))\leq\eps \,\,\,\, \textrm{for every} \,\,\,\, k\leq 0\}.$$
The (Bowen's) \emph{dynamical ball} of $x$ with radius $\eps$ is the set $$\Gamma_{\eps}(x)=W^u_{\eps}(x)\cap W^s_{\eps}(x).$$ 
We say that $f$ is \emph{expansive} if there exists $c>0$ such that $$\Gamma_c(x)=\{x\} \,\,\,\,\,\, \text{for every} \,\,\,\,\,\, x\in X.$$ We say that $f$ satisfies the \emph{local-product-structure} if for each $\eps>0$ there exists $\delta>0$ such that $$W^s_\eps(x)\cap W^u_\eps(y)\neq\emptyset \,\,\,\,\,\, \text{whenever} \,\,\,\,\,\, d(x,y)<\delta.$$ In some texts, the local-product-structure is called by canonical coordinates, as in \cite{A}. The expansive homeomorphisms satisfying the local-product-structure are called \emph{topologically hyperbolic}. 
For each expansive homeomorphism $f\colon X\to X$ of a compact metric space $X$, Fathi constructed a compatible metric $d_F$ and constants $\xi>0$ and $\lambda>1$ such that if $d_F(x,y)<\xi$, then
$$\max\{d_F(f(x),f(y)),d_F(f^{-1}(x),f^{-1}(y))\}\geq\lambda d_F(x,y).$$ An important feature of this metric is that points in the same local stable/unstable set contract exponentially when iterated. Thus, if $y\in W^s_{\xi}(x)$ and $z\in W^u_{\xi}(x)$, then 
$$d_F(f^k(y),f^k(x))\leq \lambda^{-k}d_F(y,x) \,\,\,\,\,\,\,\,\,\, \textrm{and}$$ $$d_F(f^{-k}(z),f^{-k}(x))\leq \lambda^{-k}d_F(z,x) \,\,\,\,\,\, \text{for every} \,\,\,\,\,\, k\geq0.$$ This is the reason this is usually called a \emph{hyperbolic metric}.
Artigue proved in \cite{A}*{Theorem 2.3} that there are a compatible metric $d$ and constants $\xi>0$ and $\lambda>1$ such that if $d(x,y)<\xi$, then $$\max\{d(f(x),f(y)),d(f^{-1}(x),f^{-1}(y))\}=\lambda d(x,y)$$ with an equality instead of an inequality. A hyperbolic metric with this property is called \emph{self-similar}. Artigue proved \cite{A}*{Proposition 2.12} that a self-similar hyperbolic metric is, in particular, a conformal hyperbolic distance, that is, points in the same local stable/unstable set contract exponentially with an equality instead of an inequality. More precisely, $d$ is a \emph{conformal hyperbolic distance} if there exist constants $\xi>0$ and $\lambda>1$ such that if $y\in W^s_{\xi}(x)$ and $z\in W^u_{\xi}(x)$, then
$$d(f^k(y),f^k(x))=\lambda^{-k}d(y,x) \,\,\,\,\,\,\,\, \textrm{and}$$ 
$$d(f^{-k}(z),f^{-k}(x))=\lambda^{-k}d(z,x) \,\,\,\,\,\, \text{for every} \,\,\,\,\,\, k\geq0.$$ We say that $\xi$ is an \emph{expansive constant} and that $\lambda$ is an \emph{expanding factor} of $d$.

\vspace{+0.4cm}

\hspace{-0.4cm}\textbf{Conformal structures}

\vspace{+0.4cm}

In what follows, we explain how to use a conformal hyperbolic distance $d$ to define distances on each (global) stable/unstable set where points contract conformally (with an equality instead of an inequality) when iterated. 

\begin{definition}
A \emph{conformal structure} $(d^s,d^u)$ for a homeomorphism $f\colon X\to X$ is a pair of maps $d^s\colon X^s\to\R^+$ and $d^u\colon X^u\to\R^+$, where 
$$X^s=\{(y,z)\in X\times X; y\in W^s(z)\} \,\,\,\,\,\, \text{and} \,\,\,\,\,\, X^u=\{(y,z)\in X\times X; y\in W^u(z)\},$$ satisfying: 
\begin{enumerate}
\item $d^s$ restricted to $W^s(x)\times W^s(x)$ and $d^u$ restricted to $W^u(x)\times W^u(x)$ are distances for every $x\in X$ and
\item there exists $\lambda>1$ such that if $y\in W^s(x)$ and $z\in W^u(x)$, then
$$d^s(f(y),f(x))=\lambda^{-1}d^s(y,x) \,\,\,\,\,\,\,\, \textrm{and} \,\,\,\,\,\,\,\, d^u(f^{-1}(z),f^{-1}(x))=\lambda^{-1}d^u(z,x).$$  
\end{enumerate}
\end{definition}

This definition does not assume the existence of a conformal hyperbolic distance and does not assume that $f$ is topologically hyperbolic, it defines an abstract object, the conformal structure, that will allow us to obtain information on stable/unstable holonomies. We construct below a conformal structure for each topologically hyperbolic homeomorphism using a conformal hyperbolic distance $d$.

\begin{definition}
We consider the map $d^s\colon X^s\to\R^+$ defined by
$$d^s(y,z)=\lambda^{n(y,z)}d(f^{n(y,z)}(y),f^{n(y,z)}(z)),$$
where $n(y,z)$ is the first iterate $n\in\N\cup\{0\}$ where $f^n(z)\in W^s_{\xi}(f^n(y))$, $\lambda>1$ is the expanding factor and $\xi$ is the expansive constant of the metric $d$. 
\end{definition}


\begin{remark}\label{0}
If $z\in W^s_{\xi}(y)$, then $n(y,z)=0$ and $d^s(y,z)=d(y,z)$, that is, $d^s=d$ on $\xi$-stable sets.
\end{remark}

\begin{remark}\label{first}
We note that the choice of the first iterate is not important for the definition of $d^s$ since for each $k\geq n(y,z)$ we have
\begin{eqnarray*}
d^s(y,z)&=&\lambda^{n(y,z)}d(f^{n(y,z)}(y),f^{n(y,z)}(z))\\
&=& \lambda^{k}\lambda^{-k+n(y,z)}d(f^{n(y,z)-k}(f^k(y)),f^{n(y,z)-k}(f^k(z)))\\
&=& \lambda^{k}\lambda^{-k+n(y,z)}\lambda^{-n(y,z)+k}d(f^{k}(y),f^{k}(z))\\
&=&\lambda^{k}d(f^{k}(y),f^{k}(z)).
\end{eqnarray*}
\end{remark}

\begin{proposition}\label{distance}
The map $d^s$ restricted to $W^s(x)\times W^s(x)$ is a metric for every $x\in X$.
\end{proposition}

\begin{proof}
Indeed, $d^s(y,z)\geq0$, with equality if, and only if, $y=z$; $d^s(y,z)=d^s(z,y)$ since $n(y,z)=n(z,y)$; and the triangle inequality is proved as follows: for each $w,y,z\in W^s(x)$ choose $k\in\N$ such that 
$$k\geq\max\{n(w,y), n(y,z), n(w,z)\}$$
and note that Remark \ref{first} and the fact that $d$ is a distance ensure that
\begin{eqnarray*}
d^s(w,z)&=& \lambda^{k}d(f^{k}(w),f^{k}(z))\\
&\leq& \lambda^{k}d(f^{k}(w),f^{k}(y))+\lambda^{k}d(f^{k}(y),f^{k}(z))\\
&=&d^s(w,y)+d^s(y,z).
\end{eqnarray*}
This concludes the proof.
\end{proof}

In the following proposition, we prove that $f$ acts in a conformal way on stable sets in the distance $d^s$.

\begin{proposition}\label{conformal}
If $y\in W^s(x)$, then 
$$d^s(f^k(x),f^k(y))=\lambda^{-k}d^s(x,y) \,\,\,\,\,\, \text{for every} \,\,\,\,\,\, k\geq0.$$    
\end{proposition}

\begin{proof}
Note that
$$n(f^k(x),f^k(y))=n(x,y)-k \,\,\,\,\,\, \text{for every} \,\,\,\,\,\, k\in\{0,\dots,n(x,y)\}$$
and, consequently,
\begin{eqnarray*}
d^s(f^k(x),f^k(y))&=& \lambda^{n(f^k(x),f^k(y))}d(f^{n(f^k(x),f^k(y))}(f^k(x)),f^{n(f^k(x),f^k(y))}(f^k(y)))\\
&=& \lambda^{n(x,y)-k}d(f^{n(x,y)}(x),f^{n(x,y)}(y))\\
&=&\lambda^{-k}d^s(x,y)
\end{eqnarray*}
for every $k\in\{0,\dots,n(x,y)\}$. If $k>n(x,y)$, then
\begin{eqnarray*}
d^s(f^k(x),f^k(y))&=& d^s(f^{k-n(x,y)}(f^{n(x,y)}(x)),f^{k-n(x,y)}(f^{n(x,y)}(y)))\\
&=&\lambda^{n(x,y)-k}d^s(f^{n(x,y)}(x),f^{n(x,y)}(y))\\
&=&\lambda^{-k}d^s(x,y).
\end{eqnarray*}
The second equality follows from Remark \ref{0} noting that 
$$f^{k-n(x,y)}(f^{n(x,y)}(x))\in W^s_{\xi}(f^{k-n(x,y)}(f^{n(x,y)}(y))).$$ This completes the proof.
\end{proof}

The same discussion can be done for unstable sets. In what follows, detailed explanations and proofs will be omitted since they are similar to the stable case.

\begin{definition}
We consider a map $d^u\colon X^u\to\R^+$ defined by
$$d^u(y,z)=\lambda^{n(y,z)}d(f^{-n(y,z)}(y),f^{-n(y,z)}(z)),$$
where $n(y,z)$ is the first iterate $n\in\N\cup\{0\}$ where $f^{-n}(z)\in W^u_{\xi}(f^{-n}(y))$. We use the same notation $n(y,z)$ for both stable and unstable cases and hope this causes no confusion to the reader. 
\end{definition}

As in Propositions \ref{distance} and \ref{conformal}, $(W^u(x),d^u)$ is a metric space where $f^{-1}$ acts in a conformal way. This is condensed in the following result.

\begin{proposition}\label{uconf}
The map $d^u$ is a metric in $W^u(x)\times W^u(x)$ and if $y\in W^u(x)$, then 
$$d^u(f^{-k}(x),f^{-k}(y))=\lambda^{-k}d^u(x,y) \,\,\,\,\,\, \text{for every} \,\,\,\,\,\, k\geq0.$$  
\end{proposition}

All above results prove the following:

\begin{theorem}
The pair $(d^s,d^u)$ is a conformal structure for $f$.
\end{theorem}

\begin{proof}
A direct consequence of Propositions \ref{distance}, \ref{conformal}, and \ref{uconf}.
\end{proof}

\vspace{+0.4cm}

\hspace{-0.4cm}\textbf{Pseudo-isometric local stable/unstable holonomies}

\vspace{+0.4cm}

All the above results hold for every expansive homeomorphism since they always have a self-similar hyperbolic metric, which as noted before is conformal. Assuming, in addition, that $f$ is topologically hyperbolic, an important feature of the self-similar hyperbolic metric is that local stable/unstable holonomies are pseudo-isometric. The local-product-structure ensures the existence of $\delta_0\in(0,\xi)$ and a map $\big{[}.,.\big{]}$ defined by 
$$\big{[}x,y\big{]}=W^s_{\xi}(x)\cap W^u_{\xi}(y) \,\,\,\,\,\, \text{whenever} \,\,\,\,\,\, d(x,y)<\delta_0.$$ This map is well-defined since $f$ is expansive and, consequently, $\big{[}x,y\big{]}$ is a singleton. It is proved in \cite{A}*{Theorem 2.21} that for each $\eps\in(0,\xi)$, there exists $\delta\in(0,\delta_0)$ such that
$$\max\left\{d(\big{[}x,y\big{]},x), d(\big{[}x,y\big{]},y)\right\}<(1+\eps)d(x,y)$$
whenever $d(x,y)<\delta$ (see also \cite{A}*{Remark 2.22}). For each $z\in X$, let
$$C_{\xi}(z)=\left[W^s_{\xi}(z),W^u_{\xi}(z)\right]=\{[x,y]; \,\,\, x\in W^u_{\xi}(z) \,\,\, \text{and} \,\,\, y\in W^s_{\xi}(z)\}.$$
The set $C_{\xi}(z)$ is called the \emph{product box} centered at $z$ and radius $\xi$. 
The sets 
$$\left[\{x\},W^s_{\xi}(z)\right] \,\,\,\,\,\, \text{and} \,\,\,\,\,\, \left[W^u_{\xi}(z),\{y\}\right]$$
are called the \emph{stable/unstable plaques} of the product box and we consider the natural local stable/unstable holonomy maps $\pi^s$ and $\pi^u$ between each pair of stable/unstable plaques in $C_{\xi}(z)$. It is proved in \cite{A}*{Lemma 4.3} that when $d$ is a self-similar hyperbolic metric, the local stable/unstable holonomy maps are \emph{pseudo-isometric}. This means that for each $\eps\in(0,\xi)$, there exists $\delta\in(0,\eps)$ such that if $a,b\in \left[\{x\},W^s_{\xi}(z)\right]$ satisfy $d(a,b)<\delta$, then 
$$(1-\eps)d(a,b)\leq d(\pi^s(a),\pi^s(b))\leq(1+\eps)d(a,b),$$ and analogously, if $p,q\in \left[W^u_{\xi}(z),\{y\}\right]$ satisfy $d(p,q)<\delta$, then 
$$(1-\eps)d(p,q)\leq d(\pi^u(p),\pi^u(q))\leq(1+\eps)d(p,q).$$
If a conformal structure $(d^s,d^u)$ is defined from a self-similar hyperbolic metric $d$, then we can change $d$ by $d^s$ and $d^u$ in the respective above inequalities, that is, the stable/unstable holonomies in a product box are pseudo-isometric with respect to $d^s$ and $d^u$. This is just a consequence of the fact that $d=d^s$ on local stable sets and $d=d^u$ on local unstable sets (see Remark \ref{0}).

It is also possible to obtain pseudo-isometric stable/unstable holonomies using a conformal hyperbolic distance, without the assumption of the existence of a self-similar hyperbolic metric. The following result is based in \cite{A}*{Lemma 4.3}.

\begin{prop}\label{confpseudo}
If $f\colon X\to X$ is a topologically hyperbolic homeomorphism of a compact metric space and $d$ is a conformal hyperbolic distance for $f$, then the local product boxes of $f$ are pseudo-isometric.
\end{prop}

\begin{proof}
Let $C_{\xi}(z)$ be any local product box of $f$ and consider an unstable holonomy map $\pi^u$ between any pair of unstable plaques $\left[\{x\},W^s_{\xi}(z)\right]$ and $\left[\{y\},W^s_{\xi}(z)\right]$ (the case of stable holonomy is analogous). For each $\eps\in(0,\xi)$, we choose $m_0\in\N$ such that $$\frac{2}{\lambda^{m_0-1}-2}<\eps.$$ For each $p,q\in \left[\{x\},W^s_{\xi}(z)\right]$ let $m=m(p,q)\in\N$ be such that 
$$\xi\lambda^{-m-1}<\max\{d(p,q),d(\pi^u(p),\pi^u(q))\}\leq\xi\lambda^{-m}.$$
Now choose $\delta\in(0,\eps)$ such that $m(p,q)\geq m_0$ whenever $d(p,q)<\delta$. 
Using the metric properties of $d$ and its conformal contractions on local stable/unstable sets, we obtain:
\begin{eqnarray*}
|d(p,q)-d(\pi^u(p),\pi^u(q))|&=&\lambda^{-m}|d(f^m(p),f^m(q))-d(f^m(\pi^u(p)),f^m(\pi^u(q)))|\\
&\leq&\lambda^{-m}(d(f^m(p),f^m(\pi^u(p)))+d(f^m(q),f^m(\pi^u(q))))\\
&=&\lambda^{-2m}(d(p,\pi^u(p))+d(q,\pi^u(q)))\\
&\leq&\lambda^{-2m}2\xi\\
&<&\lambda^{-m+1}2\max\{d(p,q),d(\pi^u(p),\pi^u(q))\}.
\end{eqnarray*}
Assume first that $$\max\{d(p,q),d(\pi^u(p),\pi^u(q))\}=d(\pi^u(p),\pi^u(q)).$$ 
Thus,
$$\left|1-\frac{d(p,q)}{d(\pi^u(p),\pi^u(q))}\right|<\lambda^{-m+1}2<\eps$$
and \cite{A}*{Lemma 4.2} ensures that
$$\left|1-\frac{d(\pi^u(p),\pi^u(q))}{d(p,q)}\right|<\frac{2}{\lambda^{m-1}-2}\leq\frac{2}{\lambda^{m_0-1}-2}<\eps.$$
In the case $$\max\{d(p,q),d(\pi^u(p),\pi^u(q))\}=d(p,q),$$ 
we have
$$\left|1-\frac{d(\pi^u(p),\pi^u(q))}{d(p,q)}\right|<\lambda^{-m+1}2<\eps$$
and \cite{A}*{Lemma 4.2} ensures that
$$\left|1-\frac{d(p,q)}{d(\pi^u(p),\pi^u(q))}\right|<\frac{2}{\lambda^{m-1}-2}\leq\frac{2}{\lambda^{m_0-1}-2}<\eps.$$
In both cases we have, in particular, that
$$|d(p,q)-d(\pi^u(p),\pi^u(q))|<\eps d(p,q)$$
and it follows that
\begin{eqnarray*}
d(\pi^u(p),\pi^u(q))&=&d(\pi^u(p),\pi^u(q))-d(p,q)+d(p,q)\\
&<&\eps d(p,q)+d(p,q)\\
&=&(1+\eps)d(p,q).
\end{eqnarray*}
The inequality $d^u(\pi^u(p),\pi^u(q))>(1-\eps)d^u(p,q)$ is proved in a similar way and this finishes the proof.
\end{proof}

\vspace{+0.4cm}

\hspace{-0.4cm}\textbf{Conformal stable/unstable lengths} 

\vspace{+0.4cm}

In the proof of Theorem \ref{A}, we will need to define the lengths of stable/unstable curves on general Peano continua. We can do this using the conformal structure $(d^s,d^u)$ as follows. We first observe that since $f$ is topologically hyperbolic and $X$ is a Peano continuum, that is a compact, connected, and locally connected metric space, it follows that every stable/unstable set is path connected and locally path-connected (see \cite{N}*{Theorem 4.1}).

\begin{definition}\label{length}
Let $\gamma^s\colon[a,b]\to X$ be a stable curve and $P=\{a=a_0<a_1<\dots<a_{n-1}<a_n=b\}$ be a partition of $[a,b]$. Let
$$\ell^s(\gamma^s,P)=\sum_{i=0}^{n-1}d^s(\gamma^s(a_{i+1}),\gamma^s(a_i))$$
and define the length of the curve $\gamma^s$ as
$$\ell^s(\gamma^s)=\sup_{P}\{\ell^s(\gamma^s,P)\},$$
where the supremum is taken over all partitions of $[a,b]$. We say that $\gamma^s$ has a \emph{well-defined stable length} if $0<\ell^s(\gamma^s)<+\infty$.
\end{definition}





The following proposition contains an important property we obtain from $\ell^s$ being a length.

\begin{proposition}\label{slength}
If $\gamma^s\colon[a,b]\to X$ is a stable curve with well-defined stable length and $c\in(a,b)$, then
$$\ell^s(\gamma^s)=\ell^s(\gamma^s_{[a,c]})+\ell^s(\gamma^s_{[c,b]}).$$
\end{proposition}

\begin{proof}
The proof follows classical arguments of real analysis (see for example \cite{E}*{Theorem 2, Chapter IX} for more details). We begin the proof by letting 
$$A=\{\ell^s(\gamma^s_{[a,c]},P); P \,\,\, \text{is a partition of} \,\,\, [a,c]\}$$ and $$B=\{\ell^s(\gamma^s_{[c,b]},Q); Q \,\,\, \text{is a partition of} \,\,\, [c,b]\}.$$ Thus, 
$$A+B=\{\ell^s(\gamma^s,R); R \,\,\, \text{is a partition of} \,\,\, [a,b] \,\,\, \text{containing $c$}\}$$ and
$$\ell^s(\gamma^s)=\sup(A+B)=\sup(A)+\sup(B)=\ell^s(\gamma^s_{[a,c]})+\ell^s(\gamma^s_{[c,b]}).$$ The first equality follows from the following observation: for any partition $P$ of $[a,b]$ we can add $c$ to $P$ and obtain a new partition $P'$ such that $\ell^s(\gamma^s,P')\geq\ell^s(\gamma^s,P)$, so the supremum over all partitions equals the supremum over the partitions containing $c$.
\end{proof}

The conformal properties of $d^s$ on stable sets also hold using the stable length.

\begin{proposition}\label{slengthconformal}
If $\gamma^s\colon[a,b]\to X$ is a stable curve with well-defined stable length, then
$$\ell^s(f^k(\gamma^s))=\lambda^{-k}\ell^s(\gamma^s) \,\,\,\,\,\, \text{for every} \,\,\,\,\,\, k\in\N.$$
\end{proposition}

\begin{proof}
If $P=\{a=a_0<a_1<\dots<a_{n-1}<a_n=b\}$ is a partition of $[a,b]$, then 
\begin{eqnarray*}
\ell^s(f^k(\gamma^s),P)&=&\sum_{i=0}^{n-1} d^s(f^k(\gamma^s(a_{i})),f^k(\gamma^s(a_{i+1})))\\
&=&\lambda^{-k}\sum_{i=0}^{n-1} d^s(\gamma^s(a_{i}),\gamma^s(a_{i+1}))\\
&=&\lambda^{-k}\ell^s(\gamma^s,P).
\end{eqnarray*}
The proposition follows by taking the supremum over all partitions $P$ of $[a,b]$ on both sides of this equality.
\end{proof}


We can analogously define the length $\ell^u$ of unstable curves as follows:

\begin{definition}
Let $\gamma^u\colon[a,b]\to X$ be an unstable curve and $P=\{a=a_0<a_1<\dots<a_{n-1}<a_n=b\}$ be a partition of $[a,b]$. Let
$$\ell^u(\gamma^u,P)=\sum_{i=0}^{n-1}d^u(\gamma^u(a_{i+1}),\gamma^u(a_i))$$
and define the length of the curve $\gamma^u$ as
$$\ell^u(\gamma^u)=\sup_{P}\{\ell^u(\gamma^u,P)\},$$
where the supremum is taken over all partitions of $[a,b]$. We say that $\gamma^u$ has a \emph{well-defined unstable length} if $0<\ell^u(\gamma^u)<+\infty$.
\end{definition}

As in Propositions \ref{slength} and \ref{slengthconformal}, we obtain length and conformal properties for $\ell^u$. This is condensed in the following result.

\begin{proposition}\label{lengthconformal}
If $\gamma^u\colon[a,b]\to X$ is an unstable curve with well-defined unstable length, then
$$\ell^u(f^{-k}(\gamma^u))=\lambda^{-k}\ell^u(\gamma^u) \,\,\,\,\,\, \text{for every} \,\,\,\,\,\, k\in\N$$
and if $c\in(a,b)$, then
$$\ell^u(\gamma^u)=\ell^u(\gamma^u_{[a,c]})+\ell^u(\gamma^u_{[c,b]}).$$
\end{proposition}


\begin{definition}
We say that a conformal structure $(d^s,d^u)$ induces a \emph{conformal length structure} if every stable/unstable curve has a well-defined stable/unstable length $\ell^s/\ell^u$.
\end{definition}

Using a conformal length structure, we obtain pseudo-isometric local stable/unstable holonomies inside s/u-rectangles.

\begin{definition} A \emph{s/u-rectangle} is a continuous map $g\colon[a,b]\times[c,d]\to X$ such that $g(t,.)$ is a stable curve for every $t\in[a,b]$, $g(.,s)$ is an unstable curve for every $s\in[c,d]$, and 
$$g(t,s)=g(t,.)\cap g(.,s) \,\,\,\,\,\, \text{for every} \,\,\,\,\,\, (t,s)\in[a,b]\times[c,d].$$ 
\end{definition}

Thus, stable/unstable holonomies can be infinitely extended if for every $x\in X$, $y\in W^s(x)$, and $z\in W^u(x)$, there exists a s/u-rectangle $g\colon[0,1]\times[0,1]\to X$ such that $g(0,0)=x$, $g(0,1)=y$, and $g(1,0)=z$. When the curves $g(t,.)$ and $g(.,s)$ are local stable/unstable curves, we call $g$ a local s/u-rectangle.

When a local stable/unstable curve has a well-defined stable/unstable length and is a plaque of a local s/u-rectangle, then all the local stable/unstable curves of this rectangle also have well-defined lengths, which are pseudo-isometrically related. This is the content of the following result. 

\begin{proposition}$($\emph{Pseudo-isometric local s/u-rectangles:}$)$\label{Pi}
For each $\eps\in(0,\xi)$, there exists $\delta\in(0,\eps)$ such that if $\gamma^u\colon[0,1]\to M$ is an unstable curve with well-defined unstable length satisfying $\ell^u(\gamma^u)\leq\delta$ and $\gamma^s\colon[0,1]\to M$ is a stable curve satisfying $\gamma^s(0)=\gamma^u(0)$ and
$$d^s(\gamma^s(0),\gamma^s(t))\leq\delta \,\,\,\,\,\, \text{for every} \,\,\,\,\,\, t\in[0,1],$$ then there exists a s/u-rectangle $g\colon[0,1]\times[0,1]\to X$ such that 
$$g(0,.)=\gamma^s, \,\,\,\,\,\, g(.,0)=\gamma^u,$$
$$(1-\eps)\ell^u(\gamma^u)\leq\ell^u(g(.,t))\leq(1+\eps)\ell^u(\gamma^u) \,\,\,\,\,\, \text{for every} \,\,\,\,\,\, t\in[0,1], \,\,\,\,\,\, \text{and}$$
$$d^s(g(t,s),g(t,0))\leq(1+\eps)d^s(\gamma^s(s),\gamma^s(0)) \,\,\,\,\,\, \text{for every} \,\,\,\,\,\, (t,s)\in[0,1]\times[0,1].$$
\end{proposition}

\begin{proof}
For each $\eps\in(0,\delta_0/2)$, let $\delta\in(0,\eps)$ be given by the pseudo-isometric local stable/unstable holonomies proved in Proposition \ref{confpseudo}. 
For each $(t,s)\in[0,1]\times[0,1]$ we have
\begin{eqnarray*}
d(\gamma^u(t),\gamma^s(s))&\leq& d^u(\gamma^u(t),\gamma^u(0))+d^s(\gamma^s(0),\gamma^s(s))\\
&\leq&\ell^u(\gamma^u_{[0,t]})+d^s(\gamma^s(0),\gamma^s(s))\\
&\leq&2\delta\\
&<&\delta_0.
\end{eqnarray*}
Thus, the map $g\colon[0,1]\times[0,1]\to X$ defined by $$g(t,s)=\left[\gamma^u(t),\gamma^s(s)\right]$$
is well-defined and is a local s/u-rectangle.
For each $t\in[0,1]$ and each partition $P=\{0=a_0<a_1<\dots<a_{n-1}<a_n=1\}$ of $[0,1]$ we have
\begin{eqnarray*}
\ell^u(g(.,t),P)&=&\sum_{i=0}^{n-1}d^u(g(a_{i},t),g(a_{i+1},t))\\
&\leq&\sum_{i=0}^{n-1}(1+\eps)d^u(g(a_{i},t)^u,g(a_{i+1},t)^u)\\
&=&(1+\eps)\ell^u(\gamma^u,Q),
\end{eqnarray*} 
where $Q=\{0=b_0<b_1<\dots<b_{n-1}<b_n=1\}$ is such that 
$$\gamma^u(b_i)=g(a_i,t)^u \,\,\,\,\,\, \text{for every} \,\,\,\,\,\, i\in\{0,\dots,n\}.$$
Since $\ell^u(\gamma^u,Q)\leq\ell^u(\gamma^u)$ for every partition $Q$, it follows by taking the supremum over all partitions $P$ of $[0,1]$ that $$\ell^u(g(.,t))\leq(1+\eps)\ell^u(\gamma^u).$$
Also, a similar argument ensures that
$$\ell^u(g(.,t),P)\geq\sum_{i=0}^{n-1}(1-\eps)d^u(g(a_{i},t)^u,g(a_{i+1},t)^u)=(1-\eps)\ell^u(\gamma^u,Q)$$
and we conclude that 
$$(1-\eps)\ell^u(\gamma^u)\leq\ell^u(g(.,t)).$$
The inequality for $d^s$ follows directly from the pseudo-isometric local-product-structure.
\end{proof}

\vspace{+0.4cm}

\hspace{-0.4cm}\textbf{Lengths are preserved by holonomies} 

\vspace{+0.4cm}

Now we prove Theorem \ref{A}, that is, the lengths of the stable/unstable curves in a conformal hyperbolic distance are preserved by the stable/unstable holonomies, whenever they are well-defined.

\begin{proof}[Proof of Theorem \ref{A}]
We first prove the result in a local s/u-rectangle $g$ with a well-defined length of its unstable curves $(\ell^u(g(.,t)))_{t\in[0,1]}$ in the metric $d$. We will prove that in this case we have 
$$\ell^u(g(.,s))=\ell^u(g(.,t)) \,\,\,\,\,\, \text{for every} \,\,\,\,\,\, s,t\in[0,1].$$
For $s<t\in[0,1]$, let $\pi_{st}\colon g(.,s)\to g(.,t)$ be the holonomy map from $g(.,s)$ to $g(.,t)$. For each $\eps\in(0,\xi)$, choose $\delta\in(0,\eps)$ such that if $\gamma^u\subset g(.,s)$ satisfies $\ell^u(\gamma^u)\leq\delta$, then 
$$(1-\eps)\ell^u(\gamma^u)\leq\ell^u(\pi_{st}(\gamma^u))\leq (1+\eps)\ell^u(\gamma^u).$$
Let $(\eps_n)_{n\in\N}$ be a decreasing sequence of positive numbers converging to 0 and $(\delta_n)_{n\in\N}$ be the associated sequence, chosen as above, such that if $\gamma^u\subset g(.,s)$ satisfies $\ell^u(\gamma^u)\leq\delta_n$, then 
$$(1-\eps_n)\ell^u(\gamma^u)\leq\ell^u(\pi_{st}(\gamma^u))\leq (1+\eps_n)\ell^u(\gamma^u).$$
Now, for each $n\in\N$ we consider a partition $P_n=\{0=a_0^n,\dots,a_{k_n}^n=1\}$ of $[0,1]$ such that 
$$\ell^u(g([a_i,a_{i+1}],s))\leq\delta_n \,\,\,\,\,\, \text{for every} \,\,\,\,\,\, i\in\{0,\dots,k_n-1\}.$$
Since for each $n\in\N$ we have
$$\ell^u(g(.,s))=\sum_{i=0}^{k_n-1}\ell^u(g([a_i,a_{i+1}],s)) \,\,\,\,\,\, \text{and}$$
$$\ell^u(g(.,t))=\sum_{i=0}^{k_n-1}\ell^u(\pi_{st}(g([a_i,a_{i+1}],s))),$$
it follows that
$$(1-\eps_n)\ell^u(g(.,s))\leq \ell^u(g(.,t))\leq (1+\eps_n)\ell^u(g(.,s)).$$ As $\eps_n\to0$ when $n\to\infty$, it follows that $\ell^u(g(.,s))=\ell^u(g(.,t))$. 

Now, assume that $\gamma^u\colon[0,1]\to M$ is an unstable curve with well-defined unstable length $0<\ell^u(\gamma^u)<+\infty$ and let $\gamma^s\colon[0,1]\to M$ be any stable curve such that $\gamma^s(0)=\gamma^u(0)$. We can assume that $\ell^u(\gamma^u)\leq\delta$ by iterating backwards if necessary, where $\delta\in(0,\xi)$ is given by Proposition \ref{Pi} for $\eps=\xi$. Let $\alpha_1>0$ be such that $d^s(\gamma^s(0),\gamma^s(\alpha_1))=\delta$ and
$$d^s(\gamma^s(0),\gamma^s(t))\leq\delta \,\,\,\,\,\, \text{for every} \,\,\,\,\,\, t\in[0,\alpha_1].$$
Then by Proposition \ref{Pi} and by the case proved above, there exists a s/u-rectangle $g_1\colon[0,1]\times[0,1]\to M$ such that 
$$g_1(0,.)=\gamma^s_{[0,\alpha_1]}, \,\,\,\,\,\, g_1(.,0)=\gamma^u, \,\,\,\,\,\, \text{and}$$
$$\ell^u(g_1(.,s))=\ell^u(\gamma^u) \,\,\,\,\,\, \text{for every} \,\,\,\,\,\, s\in[0,1].$$
Now, let $\alpha_2>\alpha_1$ be such that $d^s(\gamma^s(\alpha_1),\gamma^s(\alpha_2))=\delta$ and
$$d^s(\gamma^s(\alpha_1),\gamma^s(t))\leq\delta \,\,\,\,\,\, \text{for every} \,\,\,\,\,\, t\in[\alpha_1,\alpha_2].$$ 
Again, by Proposition \ref{Pi} and by the case proved above, there exists a s/u-rectangle $g_2\colon[0,1]\times[0,1]\to M$ such that 
$$g_2(0,.)=\gamma^s_{[\alpha_1,\alpha_2]}, \,\,\,\,\,\, g_2(.,0)=g_1(.,1), \,\,\,\,\,\, \text{and}$$
$$\ell^u(g_2(.,s))=\ell^u(g_2(.,0))=\ell^u(g_1(.,1))=\ell^u(\gamma^u) \,\,\,\,\,\, \text{for every} \,\,\,\,\,\, s\in[0,1].$$
We can apply this process a finite number of times and obtain s/u-rectangles $(g_n)_{n=1}^k$ and numbers $(\alpha_n)_{n=0}^k$ satisfying $\alpha_k=1$, $\alpha_0=0$, and the following for every $n\in\{1,\dots,k\}$: 
$$\alpha_{n}>\alpha_{n-1}, \,\,\,\,\,\, d^s(\gamma^s(\alpha_n),\gamma^s(\alpha_{n-1}))=\delta,$$ 
$$d^s(\gamma^s(\alpha_n),\gamma^s(t))\leq\delta \,\,\,\,\,\, \text{for every} \,\,\,\,\,\, t\in[\alpha_{n-1},\alpha_n],$$ 
$$g_n(0,.)=\gamma^s_{[\alpha_{n-1},\alpha_{n}]}, \,\,\,\,\,\, g_n(.,0)=g_{n-1}(.,1), \,\,\,\,\,\, \text{and}$$
$$\ell^u(g_n(.,s))=\ell^u(\gamma^u) \,\,\,\,\,\, \text{for every} \,\,\,\,\,\, s\in[0,1].$$ The s/u-rectangle $g\colon[0,1]\times[0,k]\to M$ defined by $g(s,t)=g_n(s,t)$ if $t\in[n-1,n]$ satisfies $g(.,0)=\gamma^u$, $g(0,.)=\gamma^s$, and $\ell^u(g(.,t))=\ell^u(\gamma^u)$ for every $t\in[0,k]$. This completes the proof.
\end{proof}

\section{Real Anosov diffeomorphisms}\label{Real}

In this section, we apply Theorem \ref{A} to the case of real Anosov diffeomorphisms, which are Anosov diffeomorphisms with a one-dimensional finest dominated splitting where contraction/expansion in each space is constant along orbits. We first deal with a simpler case, where contractions/expansions on each space are constant on the whole ambient manifold. Thus, let $f\colon M\to M$ be an Anosov diffeomorphism that admit a continuous decomposition of its tangent bundle into $Df$-invariant one-dimensional sub-spaces
$$E^s=E^s_1\oplus\dots\oplus E^s_k \,\,\,\,\,\, \text{and} \,\,\,\,\,\, E^u=E^u_1\oplus\dots\oplus E^u_l$$ such that there are 
$$0<|a_1|\leq\dots\leq|a_k|<1<|b_1|\leq\dots\leq|b_l|$$ satisfying $$\|Df(x)v\|_{f(x)}=|a_i|.\|v\|_x \,\,\,\,\,\, \text{for every} \,\,\,\,\,\, v\in E^s_i(x) \,\,\,\,\,\, \text{and}$$
$$\|Df^{-1}(x)v\|_{f^{-1}(x)}=|b_j^{-1}|.\|v\|_x \,\,\,\,\,\, \text{for every} \,\,\,\,\,\, v\in E^u_j(x),$$
where we denote by $\|.\|\colon TM\to\R$ the Riemannian metric and by $\|.\|_x$ the norm in $T_xM$.
We present below a conformal hyperbolic distance for this class of Anosov diffeomorphisms (this is based in \cite{A}*{Example 2.7}).
Let
$$\lambda=\max\{|a_i^{-1}|,|b_j|; i\in\{1,\dots,k\}, j\in\{1,\dots,l\}\}>1.$$
For each $(x,v)\in TM$, we can write $v=v^s+v^u$ where $v^s=\sum_{i=1}^kv^s_i$ with $v^s_i\in E^s_i$ and $v^u=\sum_{j=1}^lv^u_j$ with $v^u_j\in E^u_j$, and define $\rho\colon TM\to\R$ by
$$\rho(x,v)=\max\left\{\|v^s_i\|_x^{\frac{\log(\lambda)}{\log(|a_i^{-1}|)}}, \|v^u_j\|_x^{\frac{\log(\lambda)}{\log(|b_j|)}} ; \,\,i\in\{1,\dots,k\}, \,\, j\in\{1,\dots,l\}\right\}.$$

\begin{remark}
We note that $\rho(x,.)$ is not necessarily a norm in $T_xM$: if $v\in E^s_i$ and the respective exponent $$\frac{\log(\lambda)}{\log(|a_i^{-1}|)}>1,$$ then $$\rho(x,tv)=\|tv\|_{x}^{\frac{\log(\lambda)}{\log(|a_i^{-1}|)}}=|t|^{\frac{\log(\lambda)}{\log(|a_i^{-1}|)}}\|v\|_{x}^{\frac{\log(\lambda)}{\log(|a_i^{-1}|)}}\neq|t|\rho(x,v)$$
whenever $|t|\neq1$.
Also,
\begin{eqnarray*}
\rho(x,v)&=&\|v\|_{x}^{\frac{\log(\lambda)}{\log(|a_i^{-1}|)}}\\
&>&\left\|\frac{v}{2}\right\|_{x}^{\frac{\log(\lambda)}{\log(|a_i^{-1}|)}}+\left\|\frac{v}{2}\right\|_{x}^{\frac{\log(\lambda)}{\log(|a_i^{-1}|)}}\\
&=&\rho\left(x,\frac{v}{2}\right)+\rho\left(x,\frac{v}{2}\right).
\end{eqnarray*} Thus, $\rho$ does not satisfy absolute homogeneity and the triangle inequality, but it is non-negative, positive-definite, and symmetric $\rho(x,-v)=\rho(x,v)$.
\end{remark}

This leads to problems in obtaining metric properties by integrating $\rho$ along curves. We circumvent this issue considering the following map:
$$\widehat{\rho}(x, v) = \inf \left\{\mu > 0 ; \rho\left(x, \frac{v}{\mu}\right) \leq 1\right\},$$
which is homogeneous of degree 1, i.e. $\widehat{\rho}(x, \lambda v) = |\lambda| \widehat{\rho}(x, v)$. Indeed, for each $\lambda>0$ and $\mu>0$ we have $$\rho\left(x, \frac{\lambda v}{\lambda\mu}\right) = \rho\left(x, \frac{v}{\mu}\right)$$
which ensures that $\widehat{\rho}(x,\lambda v) = \lambda\widehat{\rho}(x, v)$.
Note that we can consider the case of $\lambda>0$ since $\rho(x,-v)=\rho(x,v)$. The map $\widehat{\rho}$ can be seen as the homogenization of the map $\rho$.
The metric $\|.\|$ defines the following classical Riemannian distance $$d(x,y)=\inf\left\{\int_a^b\|\gamma'(t)\|_{\gamma(t)}dt; \,\,\gamma\colon[a,b]\overset{pwC^1}{\longrightarrow} M, \gamma(a)=x, \gamma(b)=y\right\},$$
where pw$C^1$ means piecewise $C^1$, that is, there exists a partition of $[a,b]$ where the restriction of $\gamma$ to each of these pieces is regular and $C^1$. The integral in the definition of $d$ is called the length of the curve $\gamma$ and is denoted by $\ell(\gamma)$.

If $\gamma\colon[0,1]\to M$ is a pw-$C^1$ curve, then the map $t\to\rho(\gamma(t),\gamma'(t))$ is pw-continuous and, consequently, the map $t\to\widehat{\rho}(\gamma(t),\gamma'(t))$ is pw-continuous and can be integrated on $[0,1]$. Thus, we can define the following map: 
$$d_{\widehat{\rho}}(x,y)=\inf\left\{\int_a^{b}\widehat{\rho}(\gamma(t),\gamma'(t))dt; \,\,\gamma\colon[a,b]{\overset{pwC^1}{\longrightarrow}} M, \gamma(a)=x, \gamma(b)=y\right\},$$ where the infimum is taken over all pw$C^1$ curves $\gamma$ connecting $x$ and $y$. The integral in the definition of $d_{\widehat{\rho}}$ is called the $\widehat{\rho}$-length of the curve $\gamma$ and is denoted by $\ell_{\widehat{\rho}}(\gamma)$. 


\begin{proposition}\label{sametopology}
The map $d_{\widehat{\rho}}$ is a metric that is Lipschitz-equivalent to $d$.
\end{proposition}

\begin{proof}
Since $\widehat{\rho}\colon TM\to\R$ is positive-definite and continuous, the constants $$C_R = \sup_{x \in M, \|v\|_x=1} \widehat{\rho}(x, v) \,\,\,\,\,\, \text{and} \,\,\,\,\,\, C_L = \inf_{x \in M, \|v\|_x=1} \widehat{\rho}(x, v)$$ are positive and finite. We note that
$$C_L \|v\|_x \leq \widehat{\rho}(x, v) \leq C_R \|v\|_x$$
for every $x \in M$ and every $v \in T_x M$. Indeed, for each $v \in T_x M$ with $v \neq 0$, we can write $v = \|v\|_x \cdot v_0$, where $v_0 = v/\|v\|_x$ is on the unit sphere. Since $\widehat{\rho}$ is 1-homogeneous it follows that 
$$\widehat{\rho}(x, v) = \widehat{\rho}(x, \|v\|_x \cdot v_0) = \|v\|_x \cdot \widehat{\rho}(x, v_0) \leq \|v\|_x \cdot C_R$$ and also that
$$\widehat{\rho}(x, v) = \|v\|_x \cdot \widehat{\rho}(x, v_0) \geq \|v\|_x \cdot C_L.$$
Integrating the above inequalities we obtain
$$\int_a^b C_L \|\gamma'(t)\|_{\gamma(t)} dt \leq \int_a^b \widehat{\rho}(\gamma(t), \gamma'(t)) dt \leq \int_a^b C_R \|\gamma'(t)\|_{\gamma(t)} dt$$
and taking the infimum over all pw$C^1$ curves $\gamma$ connecting $x$ to $y$ we obtain $$C_L \cdot d(x,y) \leq d_{\widehat{\rho}}(x,y) \leq C_R \cdot d(x,y).$$ This proves that $d_{\widehat{\rho}}$ and $d$ are Lipschitz-equivalent. Now we prove that $d_{\widehat{\rho}}$ is a metric. The above inequalities ensure that $d_{\widehat{\rho}}$ is non-negative and that $d_{\widehat{\rho}}(x,y)=0$ if, and only if, $x=y$. The symmetry follows by observing that if $\gamma\colon[a,b]\to M$ is a pw$C^1$ curve connecting $x$ and $y$, then $\tilde{\gamma}\colon[a,b]\to M$ defined by $\tilde{\gamma}(t)=\gamma(b+a-t)$ is the reverse curve connecting $y$ and $x$ with tangent vector $\tilde{\gamma}'(t)=-\gamma'(b+a-t)$. Since $$\widehat{\rho}(\tilde{\gamma}(t),\tilde{\gamma}'(t))=\widehat{\rho}(\gamma(b+a-t),-\gamma'(b+a-t))=\widehat{\rho}(\gamma(b+a-t),\gamma'(b+a-t))$$
for every $t\in[a,b]$, it follows that
$$\int_a^{b}\widehat{\rho}(\tilde{\gamma}(t),\tilde{\gamma}'(t))dt=\int_a^{b}\widehat{\rho}(\gamma(t),\gamma'(t))dt$$
and, consequently, that $d_{\widehat{\rho}}(y,x)=d_{\widehat{\rho}}(x,y)$. To prove the triangle inequality, let $x,y,z\in M$ and $\eps>0$. Choose pw$C^1$ curves $\gamma_1\colon[a,b]\to M$ and $\gamma_2\colon[a',b']\to M$ connecting $x$ to $z$ and $z$ to $y$, respectively, such that
$$\int_a^{b}\widehat{\rho}(\gamma_1(t),\gamma_1'(t))dt<d_{\widehat{\rho}}(x,z)+\eps \,\,\,\,\,\, \text{and}$$
$$\int_{a'}^{b'}\widehat{\rho}(\gamma_2(t),\gamma_2'(t))dt<d_{\widehat{\rho}}(z,y)+\eps.$$
Their concatenation $\gamma_1\#\gamma_2\colon[a,b+b'-a']\to M$ defined by
$$\gamma_1\#\gamma_2(t)=\begin{cases}\gamma_1(t), \,\,\, t\in[a,b]\\
\gamma_2(t+a'-b), \,\,\, t\in[b,b+b'-a']\end{cases}$$
is a pw$C^1$ stable curve connecting $x$ to $y$ such that
\begin{eqnarray*}
d_{\widehat{\rho}}(x,y)&\leq& \int_a^{b+b'-a'}\widehat{\rho}(\gamma_1\#\gamma_2(t),\gamma_1\#\gamma_2'(t))dt\\
&=&\int_a^{b}\widehat{\rho}(\gamma_1\#\gamma_2(t),\gamma_1\#\gamma_2'(t))dt+\int_{b}^{b+b'-a'}\widehat{\rho}(\gamma_1\#\gamma_2(t),\gamma_1\#\gamma_2'(t))dt\\
&=&\int_a^{b}\widehat{\rho}(\gamma_1(t),\gamma_1'(t))dt+\int_{b}^{b+b'-a'}\widehat{\rho}(\gamma_2(t+a'-b)),\gamma_2'(t+a'-b)))dt\\
&=&\int_a^{b}\widehat{\rho}(\gamma_1(t),\gamma_1'(t))dt+\int_{a'}^{b'}\widehat{\rho}(\gamma_2(t),\gamma_2'(t))dt\\
&<&d_{\widehat{\rho}}(w,z)+d_{\widehat{\rho}}(z,y)+2\eps.
\end{eqnarray*}
Since this holds for every $\eps>0$, the triangle inequality follows by letting $\eps\to0$ in the last inequality. This finishes the proof.
\end{proof}

Now we prove that the Anosov diffeomorphism $f$ acts in a conformal way on stable/unstable sets using the metric $d_{\widehat{\rho}}$. We actually consider the metrics $d_{\widehat{\rho}}^s$ and $d_{\widehat{\rho}}^u$ defined on stable and unstable sets, respectively, by restricting the curves in the definition of $d_{\widehat{\rho}}$ which are stable and unstable. This is exactly the same relation between the metrics $d$ and its restrictions $d^s$ and $d^u$.

\begin{proposition}\label{drhoself}
The pair $(d_{\widehat{\rho}}^s,d_{\widehat{\rho}}^u)$ is a conformal structure.
\end{proposition}

\begin{proof}
First, note that if $x\in M$ and $v\in T_xM$, then
\begin{eqnarray*}
\rho(f(x),Df(x)v)&=&\underset{j\in\{1,\dots,l\}}{\underset{i\in\{1,\dots,k\}}{\max}}\left\{\|Df(x)v^s_i\|_{f(x)}^{\frac{\log(\lambda)}{\log(|a_i^{-1}|)}}, \|Df(x)v^u_j\|_{f(x)}^{\frac{\log(\lambda)}{\log(|b_j|)}}\right\} \\
&=&\underset{j\in\{1,\dots,l\}}{\underset{i\in\{1,\dots,k\}}{\max}}\left\{(|a_i|\|v^s_i\|_x)^{\frac{\log(\lambda)}{\log(|a_i^{-1}|)}}, (|b_j|\|v^u_j\|_x)^{\frac{\log(\lambda)}{\log(|b_j|)}}\right\} \\
&=&\underset{j\in\{1,\dots,l\}}{\underset{i\in\{1,\dots,k\}}{\max}}\left\{\lambda^{-1}\|v^s_i\|_x^{\frac{\log(\lambda)}{\log(|a_i^{-1}|)}}, \lambda\|v^u_j\|_x^{\frac{\log(\lambda)}{\log(|b_j|)}}\right\}.
\end{eqnarray*}
In particular, if $v\in E^s(x)$, then $v^u_j=0$ for every $j\in\{1,\dots,l\}$ and
$$\rho(f(x),Df(x)v)=\underset{i\in\{1,\dots,k\}}{\max}\left\{\lambda^{-1}\|v^s_i\|_x^{\frac{\log(\lambda)}{\log(|a_i^{-1}|)}}\right\}=\lambda^{-1}\rho(x,v).$$
This implies that
\begin{eqnarray*}
\widehat{\rho}(f(x),Df(x)v)&=&\inf \left\{\mu > 0 ; \rho\left(f(x), \frac{Df(x)v}{\mu}\right) \leq 1\right\}\\
&=&\inf \left\{\mu > 0 ; \lambda^{-1}\rho\left(x, \frac{v}{\mu}\right) \leq 1\right\}\\
&=&\lambda^{-1}\widehat{\rho}(x,v).
\end{eqnarray*}
Also, if $x\in M$ and $v\in T_xM$, then 
\begin{eqnarray*}
\rho(f^{-1}(x),Df^{-1}(x)v)&=&\underset{j\in\{1,\dots,l\}}{\underset{i\in\{1,\dots,k\}}{\max}}\left\{\|Df^{-1}(x)v^s_i\|_{f^{-1}(x)}^{\frac{\log(\lambda)}{\log(|a_i^{-1}|)}}, \|Df^{-1}(x)v^u_j\|_{f^{-1}(x)}^{\frac{\log(\lambda)}{\log(|b_j|)}}\right\} \\
&=&\underset{j\in\{1,\dots,l\}}{\underset{i\in\{1,\dots,k\}}{\max}}\left\{(|a_i^{-1}|\|v^s_i\|_x)^{\frac{\log(\lambda)}{\log(|a_i^{-1}|)}}, (|b_j^{-1}|\|v^u_j\|_x)^{\frac{\log(\lambda)}{\log(|b_j|)}}\right\} \\
&=&\underset{j\in\{1,\dots,l\}}{\underset{i\in\{1,\dots,k\}}{\max}}\left\{\lambda\|v^s_i\|_x^{\frac{\log(\lambda)}{\log(|a_i^{-1}|)}}, \lambda^{-1}\|v^u_j\|_x^{\frac{\log(\lambda)}{\log(|b_j|)}}\right\}.
\end{eqnarray*}
In particular, if $v\in E^u(x)$, then $v^s_i=0$ for every $i\in\{1,\dots,k\}$ and
$$\rho(f^{-1}(x),Df^{-1}(x)v)=\underset{j\in\{1,\dots,l\}}{\max}\left\{\lambda^{-1}\|v^u_j\|_x^{\frac{\log(\lambda)}{\log(|b_j|)}}\right\}=\lambda^{-1}\rho(x,v).$$
This implies that
\begin{eqnarray*}
\widehat{\rho}(f^{-1}(x),Df^{-1}(x)v)&=&\inf \left\{\mu > 0 ; \rho\left(f^{-1}(x), \frac{Df^{-1}(x)v}{\mu}\right) \leq 1\right\}\\
&=&\inf \left\{\mu > 0 ; \lambda^{-1}\rho\left(x, \frac{v}{\mu}\right) \leq 1\right\}\\
&=&\lambda^{-1}\widehat{\rho}(x,v).
\end{eqnarray*}
Now, if $y\in W^s(x)$ and $\gamma\colon[a,b]\to M$ is any pw$C^1$ stable curve connecting $x$ to $y$, then 
\begin{eqnarray*}
\int_a^{b}\widehat{\rho}(f\circ\gamma(t),(f\circ\gamma)'(t))dt&=&\int_a^{b}\widehat{\rho}(f\circ\gamma(t),Df(\gamma(t))\gamma'(t))dt\\
&=&\lambda^{-1}\int_a^{b}\widehat{\rho}(\gamma(t),\gamma'(t))dt,
\end{eqnarray*}
which ensures that
$$d_{\widehat{\rho}}^s(f(x),f(y))=\lambda^{-1} d_{\widehat{\rho}}^s(x,y).$$ 
An analogous argument ensures that if $y\in W^u(x)$, then $$d_{\widehat{\rho}}^u(f^{-1}(x),f^{-1}(y))=\lambda^{-1} d_{\widehat{\rho}}^u(x,y).$$
This proves that $(d_{\widehat{\rho}}^s,d_{\widehat{\rho}}^u)$ is a conformal structure and finishes the proof.
\end{proof}


In the following result, we prove that this conformal structure induces a length structure on pw$C^1$ stable/unstable curves given by the integral of the map $\widehat{\rho}$.

\begin{proposition}
If $\gamma\colon[a,b]\to M$ is a pw$C^1$ stable curve, then its stable $\widehat{\rho}$-length $\ell_{\widehat{\rho}}^s(\gamma)$ is well-defined (as in Definition \ref{length} using the metric $d^s_{\widehat{\rho}}$) and $$\ell_{\widehat{\rho}}^s(\gamma)=\int_a^b\widehat{\rho}(\gamma(t),\gamma'(t))dt.$$ Also, if $\eta\colon[a,b]\to M$ is a pw$C^1$ unstable curve, then its unstable $\widehat{\rho}$-length $\ell_{\widehat{\rho}}^u(\eta)$ is well-defined and $$\ell_{\widehat{\rho}}^u(\eta)=\int_a^b\widehat{\rho}(\eta(t),\eta'(t))dt.$$
\end{proposition}

\begin{proof}
If $P=\{a=a_0<a_1<\dots<a_{n-1}<a_n=b\}$ is a partition of $[a,b]$, then
$$\ell_{\widehat{\rho}}^s(\gamma,P):=\sum_{i=0}^nd_{\widehat{\rho}}^s(\gamma(a_{i+1}),\gamma(a_i))\leq\int_a^b\widehat{\rho}(\gamma(t),\gamma'(t))dt<+\infty.$$ Since this integral does not depend on $P$, it follows that $\ell_{\widehat{\rho}}^s(\gamma)<+\infty$. Also, since $\widehat{\rho}$ is positive-definite, it follows that $\ell_{\widehat{\rho}}^s(\gamma)>0$, which implies that $\ell_{\widehat{\rho}}^s(\gamma)$ is well-defined. The same holds for the unstable $\widehat{\rho}$-length $\ell^u_{\widehat{\rho}}(\eta)$ of pw$C^1$ unstable curves. 

Now we prove the equalities in the statement. First, we use that $\gamma$ is pw$C^1$ to choose a partition $Q=\{a=b_0<b_1<\dots<b_{n-1}<b_n=b\}$ where $\gamma$ is $C^1$ on each interval $[b_i,b_{i+1}]$. We prove that for each $\eps>0$, we can choose $\delta\in(0,\eps)$ such that if $0<s-r<\delta$ and $s,r\in[b_i,b_{i+1}]$, then $$\left|\int_r^s \widehat{\rho}(\gamma(t),\gamma'(t))dt  - d_{\widehat{\rho}}^s(\gamma(r),\gamma(s))\right|<\eps.$$
We prove this by contradiction, assuming the existence of $\eps>0$ such that for each $n\in\N$ there exist $0<s_n-r_n<\frac{1}{n}$ with $s_n,r_n\in[b_i,b_{i+1}]$ such that
$$\left|\int_{r_n}^{s_n} \widehat{\rho}(\gamma(t),\gamma'(t))dt  - d_{\widehat{\rho}}^s(\gamma(r_n),\gamma(s_n))\right|\geq\eps.$$
But since the map $t\to\widehat{\rho}(\gamma(t),\gamma'(t))$ is continuous on $[b_i,b_{i+1}]$ and $d_{\widehat{\rho}}^s$ is a metric, it follows that $$\int_{r_n}^{s_n} \widehat{\rho}(\gamma(t),\gamma'(t))dt\to0 \,\,\,\,\,\, \text{and} \,\,\,\,\,\, d_{\widehat{\rho}}^s(\gamma(r_n),\gamma(s_n))\to0$$
when $n\to\infty$, contradicting the above inequality. This gives us a control in the above difference on each small interval of a partition of $[a,b]$ into $n$ intervals, but when we sum over all intervals of the partition, the difference will be less than $n\eps$. We need a global control on the interval $[a,b]$ which can be obtained proving the following: for each $\eps>0$, we can choose $\delta\in(0,\eps)$ such that if $0<s-r<\delta$ and $s,r\in[b_i,b_{i+1}]$, then $$\left|\int_r^s \widehat{\rho}(\gamma(t),\gamma'(t))dt  - d_{\widehat{\rho}}^s(\gamma(r),\gamma(s))\right|<\frac{\eps}{b-a}(s-r).$$ This ensures that if $P=\{a=a_0<a_1<\dots<a_{n-1}<a_n=b\}$ is a partition of $[a,b]$ refining $Q$ such that $|a_{i+1}-a_i|<\delta$ for every $i\in\{0,\dots,n-1\}$, then \begin{eqnarray*}
\ell_{\widehat{\rho}}^s(\gamma,P)&=&
\sum_{i=0}^{n-1}d_{\widehat{\rho}}^s(\gamma(a_{i+1}),\gamma(a_i))\\
&\geq& \sum_{i=0}^{n-1}\int_{a_i}^{a_{i+1}}\widehat{\rho}(\gamma(t),\gamma'(t))dt -\frac{\eps}{b-a}\sum_{i=0}^{n-1}(a_{i+1}-a_i)\\
&=&\int_a^b\widehat{\rho}(\gamma(t),\gamma'(t))dt-\eps
\end{eqnarray*}
ensuring that $\ell^s_{\widehat{\rho}}(\gamma)=\int_a^b\widehat{\rho}(\gamma(t),\gamma'(t))dt$. To prove the desired inequality, we proceed with the following estimates. 
First, we choose $\delta\in(0,\eps)$ such that if $0<s-r<\delta$ and $s,r\in[b_i,b_{i+1}]$, then $$|\widehat{\rho}(\gamma(t),\gamma'(t))-\widehat{\rho}(\gamma(r),\gamma'(r))|\leq\frac{\eps}{2(b-a)} \,\,\,\,\,\, \text{for every} \,\,\,\,\,\, t\in[r,s].$$ Thus,
$$\int_r^s \widehat{\rho}(\gamma(t),\gamma'(t))dt\leq \left(\widehat{\rho}(\gamma(r),\gamma'(r))+\frac{\eps}{2(b-a)}\right)(s-r).$$
Now, we estimate $d_{\widehat{\rho}}^s(\gamma(r),\gamma(s))$ from below as follows. 
Note that if $\beta\colon[r,s]\to M$ is a pw$C^1$ stable curve with $\beta(r)=\gamma(r)$, $\beta(s)=\gamma(s)$, and $$|\widehat{\rho}(\beta(t),\beta'(t))-\widehat{\rho}(\gamma(r),\gamma'(r))|\leq \frac{\eps}{2(b-a)} \,\,\,\,\,\, \text{for every} \,\,\,\,\,\, t\in[r,s],$$ then
\begin{eqnarray*}\int_r^s \widehat{\rho}(\beta(t),\beta'(t))dt&\geq& \left(
\widehat{\rho}(\gamma(r),\gamma'(r))-\frac{\eps}{2(b-a)}\right)(s-r)\\
&\geq& \int_r^s \widehat{\rho}(\gamma(t),\gamma'(t))dt-\frac{\eps}{(b-a)}(s-r).
\end{eqnarray*}
Since this inequality holds for any $\beta$ as above, it follows that
$$d^s_{\widehat{\rho}}(\gamma(r),\gamma(s))\geq\int_r^s \widehat{\rho}(\gamma(t),\gamma'(t))dt-\frac{\eps}{(b-a)}(s-r).$$
The above inequalities imply that
$$\left|\int_r^s \widehat{\rho}(\gamma(t),\gamma'(t))dt  - d_{\widehat{\rho}}^s(\gamma(r),\gamma(s))\right|<\frac{\eps}{b-a}(s-r).$$
This proves the equality in the stable case and a similar argument for the unstable case finishes the proof.
\end{proof}

Now we explain how to use the conformal structure $(d_{\widehat{\rho}}^s,d_{\widehat{\rho}}^u)$ to define a conformal hyperbolic distance $d_{\rho}$ on $M$.
For each pair of points $(x,y)$ in $M$, we can consider a pw$C^1$ curve connecting $x$ and $y$ that is composed of alternating stable/unstable segments, these curves will be called s/u-curves. More precisely, $\gamma\colon[a,b]\to M$ is a \emph{s/u-curve} if there is a partition $a=t_0<t_1<\dots<t_{k-1}<t_k=b$ of $[a,b]$ such that $\gamma_{|[t_{i-1},t_{i}]}$ is either stable or unstable for every $i\in\{1,\dots,k\}$, and if $\gamma_{|[t_{i-1},t_{i}]}$ is stable (unstable), then $\gamma_{|[t_{i},t_{i+1}]}$ is unstable (stable) for every $i\in\{1,\dots,k-1\}$. For each s/u-curve $\gamma$, let $$\ell_{\rho}(\gamma)=\sum_{i=0}^{k-1}\ell^{\sigma(i)}_{\widehat{\rho}}(\gamma_{|[t_i,t_{i+1}]}),$$
where $\sigma(i)=s$ if $\gamma_{|[t_{i},t_{i+1}]}$ is stable and $\sigma(i)=u$ if $\gamma_{|[t_{i},t_{i+1}]}$ is unstable. Note that if $\gamma$ is a stable (unstable) curve, then $\ell_{\rho}(\gamma)=\ell^s_{\widehat{\rho}}(\gamma)$ ($\ell_{\rho}(\gamma)=\ell^u_{\widehat{\rho}}(\gamma)$). Now define $d_{\rho}\colon X\times X\to\R^+$ by
$$d_{\rho}(x,y)=\inf\{\ell_{\rho}(\gamma); \,\, \gamma\colon[a,b]\to X \,\,\text{is a s/u-curve}, \,\, \gamma(a)=x, \gamma(b)=y\}.$$
The fact that $d_{\rho}$ is a metric follows from calculations similar to the proof that  $d_{\widehat{\rho}}^s$ and $d_{\widehat{\rho}}^u$ are metrics. 


\begin{proposition}
There exists $\eps>0$ such that if $y\in W^s_{\eps}(x)$, then $d_{\rho}(x,y)=d_{\widehat{\rho}}^s(x,y)$. Similarly, if $y\in W^u_{\eps}(x)$, then $d_{\rho}(x,y)=d_{\widehat{\rho}}^u(x,y)$. In particular, $d_{\rho}$ is a conformal hyperbolic distance.
\end{proposition}

\begin{proof}
The inequality $d_{\rho}(x,y)\leq d_{\rho}^s(x,y)$ is clearly satisfied for every $y\in W^s(x)$ and $x\in M$ since stable curves are, in particular, s/u-curves. To prove the equality on local stable/unstable manifolds, we first choose a small $\eps_0>0$ such that $W^s_{\eps_0}(x)$ and $W^u_{\eps_0}(x)$ are immersed submanifolds for every $x\in M$. Let $U_x^s, U^u_x\subset M$ be tubular neighborhoods around $W^s_{\eps_0}(x)$ and $W^u_{\eps_0}(x)$, respectively, and $\eps\in(0,\eps_0)$ be such that $C_{\eps}(x)\subset U_x^s\cap U^u_x$ for every $x\in M$, where $C_{\eps}(x)$ is the local product box centered at $x$ and radius $\eps$. Consider the natural projections $\pi^s_x\colon U^s_x\to W^s_{\eps_0}(x)$ and $\pi^u_x\colon U^u_x\to W^u_{\eps_0}(x)$. Note that the derivative of these projections are projections from $T_pM=\mathbb{E}^s_p\oplus\mathbb{E}^u_p$ to $\mathbb{E}^s_{\pi^s_x(p)}$ and $\mathbb{E}^u_{\pi^u_x(p)}$, respectively, and satisfy:
$$\|D\pi^s_x(p)v\|_{\pi^s_x(p)} \leq\|v\|_p\,\,\,\,\,\, \text{and}\,\,\,\,\,\, \|D\pi^u_x(p)v\|_{\pi^u_x(p)}\leq\|v\|_p.$$
These inequalities and the definition of $\rho$ ensure that
$$\rho(\pi^s_x(p),D\pi^s_x(p)v)\leq\rho(p,v) \,\,\,\,\,\, \text{and}\,\,\,\,\,\, \rho(\pi^u_x(p),D\pi^u_x(p)v)\leq\rho(p,v).$$
Similar inequalities also hold for the map $\widehat{\rho}$.
Now let $y\in W^s_{\eps}(x)$ and $\gamma\colon[a,b]\to M$ be an s/u-curve connecting $x$ and $y$ contained in $C_{\eps}(x)$. Consider the curve $\gamma_s=\pi^s_x\circ\gamma$, which is a stable curve connecting $x$ and $y$ contained in $W^s_{\eps_0}(x)$, and note that
\begin{eqnarray*}
\ell_{\rho}(\gamma)&=&\sum_{i=0}^{k-1}\ell^{\sigma(i)}_{\widehat{\rho}}(\gamma_{|[t_i,t_{i+1}]})\\
&=&\sum_{i=0}^{k-1}\int_{t_i}^{t_{i+1}}\widehat{\rho}(\gamma(t),\gamma'(t))dt\\
&\geq&\sum_{i=0}^{k-1}\int_{t_i}^{t_{i+1}}\widehat{\rho}(\pi^{s}_{x}(\gamma(t)),D\pi^{s}_{x}(\gamma(t))\gamma'(t))dt\\
&=&\int_{a}^{b}\widehat{\rho}(\pi^{s}_{x}(\gamma(t)),D\pi^{s}_{x}(\gamma(t))\gamma'(t))dt\\
&=&\ell_{\widehat{\rho}}^s(\gamma_s).
\end{eqnarray*}
This ensures that the infimum taken over all stable curves connecting $x$ and $y$ and the infimum taken over all s/u-curves connecting $x$ and $y$ are equal, that is, $d_{\rho}(x,y)=d_{\widehat{\rho}}^s(x,y)$. A similar calculation proves the unstable equality. The fact that $d_{\rho}$ is a conformal hyperbolic distance follows from $(d^s_{\widehat{\rho}},d^u_{\widehat{\rho}})$ being a conformal structure (see Proposition \ref{drhoself}).
\end{proof}

Now we generalize this construction considering contractions/expansions that are not necessarily constant over the ambient space but are constant along orbits, i.e., we assume the existence of continuous functions $a_i, b_j\colon M\to\R$ satisfying the following for every $i\in\{1,\dots,k\}$, $j\in\{1,\dots,l\}$, and $x\in M$:
$$a_i(f(x))=a_i(x), \,\,\,\,\,\, b_j(f(x))=b_j(x),$$
$$0<|a_1(x)|\leq\dots\leq|a_k(x)|<1<|b_1(x)|\leq\dots\leq|b_l(x)|,$$
$$\|Df(x)v\|_{f(x)}=|a_i(x)|.\|v\|_x \,\,\,\,\,\, \text{for every} \,\,\,\,\,\, v\in E^s_i(x) \,\,\,\,\,\, \text{and}$$
$$\|Df^{-1}(x)v\|_{f^{-1}(x)}=|b_j(x)^{-1}|.\|v\|_x \,\,\,\,\,\, \text{for every} \,\,\,\,\,\, v\in E^u_j(x).$$
This is what we mean by an Anosov diffeomorphism with a one-dimensional finest dominated splitting with contraction/expansion in each space is constant along orbits, which is our definition of real Anosov diffeomorphism.

Since the functions $a_i, b_j$ are not necessarily constant, we will perform a modification of the Riemaniann metric $\|.\|$ to ensure that the contraction/expansion on each space $E^s_i, E^u_j$ is constant over the ambient manifold. The modification is done as follows: for each $i\in\{1,\dots,k\}$ and $j\in\{1,\dots,l\}$, let
$$\alpha_i=\min_{x\in M}|a_i(x)| \,\,\,\,\,\, \text{and} \,\,\,\,\,\, \beta_j=\max_{x\in M}|b_j(x)|$$
and define the map $\kappa\colon TM\to\R^{+}$ by
$$\kappa(x,v)=\max\left\{\|v^s_i\|_x^{\frac{\log(\alpha_i)}{\log(|a_i(x)|)}}, \|v^u_j\|_x^{\frac{\log(\beta_j)}{\log(|b_j(x)|)}} ; \,\,i\in\{1,\dots,k\}, \,\, j\in\{1,\dots,l\}\right\}.$$
The following result ensures that this is exactly the modification we want.

\begin{proposition}
If $x\in M$ and $v\in E^s_i(x)$, then $$\kappa(f(x),Df(x)v)=\alpha_i\kappa(x,v),$$ and if $w\in E^u_j(x)$, then $$\kappa(f(x),Df(x)w)=\beta_j\kappa(x,v).$$
\end{proposition}

\begin{proof}
First, note that
\begin{eqnarray*}
\kappa(f(x),Df(x)v)&=&\underset{j\in\{1,\dots,l\}}{\underset{i\in\{1,\dots,k\}}{\max}}\left\{\|Df(x)v^s_i\|_{f(x)}^{\frac{\log(\alpha_i)}{\log(|a_i(f(x))|)}}, \|Df(x)v^u_j\|_{f(x)}^{\frac{\log(\beta_j)}{\log(|b_j(f(x))|)}}\right\} \\
&=&\underset{j\in\{1,\dots,l\}}{\underset{i\in\{1,\dots,k\}}{\max}}\left\{(|a_i(x)|\|v^s_i\|_x)^{\frac{\log(\alpha_i)}{\log(|a_i(x)|)}}, (|b_j|\|v^u_j\|_x)^{\frac{\log(\beta_j)}{\log(|b_j(x)|)}}\right\} \\
&=&\underset{j\in\{1,\dots,l\}}{\underset{i\in\{1,\dots,k\}}{\max}}\left\{\alpha_i\|v^s_i\|_x^{\frac{\log(\alpha_i)}{\log(|a_i(x)|)}}, \beta_j\|v^u_j\|_x^{\frac{\log(\beta_j)}{\log(|b_j(x)|)}}\right\}.
\end{eqnarray*}
Thus, if $v\in E^s_i(x)$, then
$$\kappa(f(x),Df(x)v)=\alpha_i\|v\|_x^{\frac{\log(\alpha_i)}{\log(|a_i(x)|)}}=\alpha_i\kappa(x,v)$$
and if $w\in E^u_j$, then
$$\kappa(f(x),Df(x)w)=\beta_j\|w\|_x^{\frac{\log(\beta_j)}{\log(|b_j(x)|)}}=\beta_j\kappa(x,v).$$
This completes the proof.
\end{proof}

Now that we have uniform contraction/expansion rates $\alpha_i, \beta_j$ on the spaces $E^s_i, E^u_j$ over the ambient manifold satisfying
$$0<\alpha_1\leq\dots\leq\alpha_k<1<\beta_1\leq\dots\leq\beta_l,$$
we apply the modification we have made before as follows: let $$\lambda=\max\{\alpha_i^{-1},\beta_j; i\in\{1,\dots,k\}, j\in\{1,\dots,l\}\}>1$$ and define $\rho\colon TM\to\R$ by
$$\rho(x,v)=\max\left\{\kappa(x,v^s_i)^{\frac{\log(\lambda)}{\log(\alpha_i^{-1})}}, \kappa(x,v^u_j)^{\frac{\log(\lambda)}{\log(\beta_j)}} ; \,\,i\in\{1,\dots,k\}, \,\, j\in\{1,\dots,l\}\right\}.$$
Since $$\kappa(x,v^s_i)=\|v^s_i\|_x^{\frac{\log(\alpha_i)}{\log(|a_i(x)|)}} \,\,\,\,\,\, \text{and} \,\,\,\,\,\, \kappa(x,v^u_j)=\|v^u_j\|_x^{\frac{\log(\beta_j)}{\log(|b_j(x)|)}},$$
it follows that
$$\kappa(x,v^s_i)^{\frac{\log(\lambda)}{\log(\alpha_i^{-1})}}=(\|v^s_i\|_x^{\frac{\log(\alpha_i)}{\log(|a_i(x)|)}})^{\frac{\log(\lambda)}{\log(\alpha_i^{-1})}}=\|v^s_i\|_x^{\frac{\log(\lambda)}{\log(|a_i(x)^{-1}|)}}$$
and
$$\kappa(x,v^u_j)^{\frac{\log(\lambda)}{\log(\beta_j)}}=(\|v^u_j\|_x^{\frac{\log(\beta_j)}{\log(|b_j(x)|)}})^{\frac{\log(\lambda)}{\log(\beta_j)}}=\|v^u_j\|_x^{\frac{\log(\lambda)}{\log(|b_j(x)|)}}.$$
This ensures that
$$\rho(x,v)=\max\left\{\|v^s_i\|_x^{\frac{\log(\lambda)}{\log(|a_i(x)^{-1}|)}}, \|v^u_j\|_x^{\frac{\log(\lambda)}{\log(|b_j(x)|)}}; \,\, i\in\{1,\dots,k\}, \,\, j\in\{1,\dots,l\}\right\}.$$
We follow the steps above and conclude that $\rho$ induces a conformal length structure $(d_{\hat\rho}^s,d_{\hat\rho}^u)$ which, in turn, defines a conformal hyperbolic distance $d_{\rho}$ with well-defined length of pw$C^1$ stable/unstable curves by integrating $\hat\rho$ and taking the infimum along s/u-curves. 

\begin{proof}[Proof of Theorem \ref{thmreal}]
Since $d_{\rho}$ is a conformal hyperbolic distance such that every pw$C^1$ stable/unstable curve has a well-defined length, it follows from Theorem \ref{A} that every pw$C^1$ stable/unstable curve has a globally defined holonomy. This implies that stable/unstable holonomies can be infinitely extended from where transitivity follows. Indeed, it can be proved that $W^s(x)\cap W^u(y)\neq\emptyset$ for every $x,y\in X$ (as noted in \cite{B}*{Remark 4}) and this implies that the whole space is a single transitive basic set (see \cite{S}*{Theorem 6.2}).
\end{proof}

\vspace{+0.4cm}

{\em B. Carvalho}
\vspace{0.2cm}

\noindent

National Laboratory for Scientific Computing – LNCC/MCTI 

Av. Getúlio Vargas 333, CEP 25651-070, 

Petrópolis – RJ, Brazil

\vspace{0.2cm}

\email{bmcarvalho@lncc.br}
\vspace{1.5cm}


\begin{thebibliography}{99}



\bibitem{A} A. Artigue. Self-similar hyperbolicity. \emph{Ergod. Th. \& Dynam. Sys}. (2018) \textbf{38} 2422–2446.


\bibitem{B} M. I. Brin. Non-wandering points of Anosov diffeomorphisms. \emph{Ast\'erisque} \textbf{49} (1977) 11-18.


\bibitem{E} E. Lages Lima. \emph{Curso de An\'alise}. Vol. 1. S\'etima ediç{c}\~ao. IMPA (1976).



\bibitem{F} A. Fathi. Expansivity, hyperbolicity and Hausdorff dimension.
\emph{Comm. Math. Phys.} \textbf{126} (1989) 249–262.

\bibitem{F1} J. Franks. \emph{Anosov diffeomorphisms}. Global Analysis (Proc. Sympos. Pure Math., Vol. XIV, Berkeley, Calif., 1968) Amer. Math. Soc., Providence, R.I., 1970, pp. 61-93.

\bibitem{F2} J. Franks. Anosov diffeomorphisms on tori. \emph{Trans. Amer. Math. Soc}. \textbf{145} (1969) 117-124.

 



\bibitem{HPS} M.W. Hirsch, C.C. Pugh, M. Shub. \emph{Invariant Manifolds}. Lecture Notes in Mathematics \textbf{583} (1977) Springer-Verlag Berlin Heidelberg.

\bibitem{KK} V. Kleptsyn, Y. Kudryashov. A curve in the unstable foliation of an Anosov diffeomorphism with globally defined holonomy. \emph{Ergodic Theory and Dynamical Systems}. \textbf{35}, no.3, (2015) 935-943. 



\bibitem{N} V. Nekrashevych. Locally connected Smale spaces, pinched spectrum, and infra-nilmanifolds. \emph{Advances in Mathematics} \textbf{374} (2020) 107385.

\bibitem{PSW} C. Pugh, M. Shub, A. Wilkinson. Hölder foliations. \emph{Duke Math. J}. \textbf{86}, no.3, (1997) 517-546.

\bibitem{S} S. Smale. Differentiable dynamical system. \emph{Bull. Amer. Math. Soc}. \textbf{73} (1967) 747-817.


\end{thebibliography}
\end{document}